\documentclass[11pt]{article}

\usepackage{amssymb}
\usepackage{amsthm}
\usepackage{amsmath}
\newtheorem{theorem}{Theorem}[section]
\newtheorem{remark}{Remark}[section]
\RequirePackage{algorithm}
\RequirePackage[noend]{algpseudocode}
\usepackage{xcolor}
\usepackage{silence}
\usepackage{hyperref}
\newtheorem{prop}{Proposition}[section]
\newtheorem{lem}{Lemma}[section]

\def\ni{\noindent}

\def\ni{\noindent}

\begin{document}

% \begin{frontmatter}

% \title{Szeg\H{o} Theorem for Operator Orthogonal Polynomials}

% \author[inst2]{Nicholas H. Bingham}

% \cortext[cor1]{Corresponding author.}

% \author[inst1]{Badr Missaoui\corref{cor1}}

% \affiliation[inst2]{organization={Department of Mathematics, Imperial College London},%Department and Organization
%             addressline={e-mail: n.bingham@ic.ac.uk}, 
% }
% \affiliation[inst1]{organization={Moroccan Center For Game Theory, UM6P},%Department and Organization
%             e-mail={e-mail: badr.missaoui@um6p.ma}}

            	\renewcommand{\thefootnote}{\arabic{footnote}}
	%$^{,}$
	\begin{center}
		{\Large \textbf{ Szeg\H{o} Theorem for Operator Orthogonal Polynomials}} \\[0pt]
		~\\[0pt] Badr Missaoui \footnote[1]{Mohammed VI Polytechnic University, Morocco. E-mail: \texttt{badr.missaoui@um6p.com}}, Nicholas H. Bingham{\footnote[2]{Department of Mathematics, Imperial College London, UK.
				E-mail: \texttt{n.bingham@ic.ac.uk}}}
		\\[0pt]

		~\\[0pt]
	\end{center}

\begin{abstract}
We develop a Szeg\H{o} theory for operator-valued measures on the unit circle, extending the framework of matrix orthogonal polynomials to the infinite-dimensional setting. Our approach yields a new and direct proof of the operator Szeg\H{o} limit theorem based solely on orthogonal-polynomial theory.
\end{abstract}


\section{Introduction}
The classical Szeg\H{o} limit theorem, originally proven in \cite{Sze}, describes the asymptotic behavior of Toeplitz determinants and plays a fundamental role in the spectral theory of Toeplitz operators and orthogonal polynomials on the unit circle. Specifically, if $\mu$ is a probability measure on the unit circle such that $\log \mu' \in L^{1}$, then  
\begin{equation}
\lim_{n \rightarrow{ \infty}}\frac{\mathrm{det}(T_{n})}{\mathrm{det}(T_{n-1})}=\prod_0^{\infty} (I - {\alpha}_k {\alpha}_k^{\dag})
= \exp \int\ \log \mu'(\theta) d \theta /2 \pi ,
\end{equation}
where $T_n=[\mu_{i-j}]_{i,j=0}^n$ and $\{\mu_k\}_{k \in \mathbb{Z}}$ are the Fourier coefficients of $\mu$, $\alpha_k$ is the \textit{k-th} Verblunsky coefficient associated with $\mu$, and ${\alpha}_k^{\dag}$ is its Hermitian conjugate. This result, deeply connected to orthogonal polynomials, has been extensively generalized in various settings, including matrix-valued orthogonal polynomials \cite{DerHK,DelGK}.  

In this paper, we extend this framework to the setting of operator orthogonal polynomials, providing a Szeg\H{o}-type limit theorem in infinite-dimensional Hilbert spaces. The study of operator-valued polynomials is not merely of theoretical interest; it has significant implications in prediction theory, particularly in the analysis of continuous stochastic processes and large collections of time series \cite{Bos, DugBM}. The use of operator-valued functions and orthogonal polynomials in prediction theory dates back to the foundational works of Mandrekar and Salehi \cite{ManS} and Weron \cite{Wer}, who developed the theory of dilation and shift operators for nonstationary processes. A key concept in this area is the \textit{Kolmogorov decomposition}, which expresses a positive definite kernel as the inner product of two operator-valued sequences, enabling the analysis of nonstationary processes through dilation and shift operators.  \\

Operator orthogonal polynomials extend the concept of matrix-valued \textit{orthogonal} polynomials by considering polynomials whose coefficients are bounded operators on a Hilbert space rather than scalars or matrices. Formally, let $\mathcal{H}$ be a separable Hilbert space, and let $\mathcal{L}(\mathcal{H})$ denote the space of bounded linear operators on $\mathcal{H}$. Given a positive, self-adjoint operator-valued measure $\mu$, we consider a sequence of operator-valued polynomials $P_n(z)$, defined on the unit circle $\mathbb{T}$, with  
\[
P_n(z) = \sum_{k=0}^{n} a_k z^k, \quad a_k \in \mathcal{L}(\mathcal{H}).
\]  
These polynomials are said to be orthogonal with respect to $\mu$ if  
\begin{align*}  
\langle P_n(z), P_m(z) \rangle_{\mu} = 0, \quad \text{for } n \neq m,  
\end{align*}  
where the inner product is defined as  
\[
\langle f, g \rangle_{\mu} := \int_{\mathbb{T}} f(z)^{\dag} d\mu(z) g(z).
\]  
Here, $\mathbb{T}$ denotes the unit circle in the complex plane, and $\mu$ is a non-trivial operator-valued measure, i.e. one not supported on a finite set. The study of such polynomials is closely related to dilation theory, which seeks to extend a given operator to a larger space while preserving its essential properties. This theory, pioneered by Sz.-Nagy and Foia\c{s} \cite{SzNF} and further developed by Nikolskii \cite{Nik2}, has been applied to operator-valued functions by Gorniak and Weron \cite{GorW} and Weron \cite{Wer}.  \\

The generalization of the Szeg\H{o} limit theorem to operator settings have previously been established. Gohberg and Kaashoek \cite{GohK} provided an operator generalization using factorization arguments and Schur complement techniques, focusing on a class of positive block Toeplitz operators with Hilbert-Schmidt entries. Their results relied on the second regularized and Perelson determinants in place of the classical determinant. Subsequent research by B\"ottcher and Silbermann \cite{BotS} removed the self-adjointness requirement on the block Toeplitz matrices and relaxed the smoothness conditions imposed by Gohberg and Kaashoek. \\ 

The novelty of the present work lies instead in the method of proof and structural viewpoint. We provide a derivation the operator Szeg\H{o} limit theorem that relies exclusively on operator-valued orthogonal polynomials on the unit circle. Our method makes use of the moments of the operator-valued measure to represent operator orthogonal polynomials in terms of Schur complements (see, e.g., \cite{Zha}). The core of our method closely follows the matrix-valued results developed in Simon's monograph \cite[Section 2.13]{Sim2}, as well as in \cite{DamPS} and \cite{DerHK}, allowing a direct extension of the matrix theory to the operator setting. This leads naturally to operator analogues of several classical results, including:
\begin{itemize}
    \item Classical recursion relations for operator orthogonal polynomials,
    \item Operator versions of the kernel polynomials and Christoffel-Darboux formulas, and
    \item A new perspective on the Bernstein-Szeg\H{o} approximation in the operator setting.
    \item A new proof of the operator Szeg\H{o} limit theorem. Our decisive result (Theorems 7.1 and 7.2) provides an operator analogue of the classical Szeg\H{o} limit theorem.
\end{itemize}

The structure of this paper is as follows: Section 2 introduces the fundamental concepts and notation essential for our analysis of operator orthogonal polynomials. In Section 3, we define and establish the basic properties of operator-valued counterparts to matrix-valued orthogonal polynomials on the unit circle. Sections 4 through 6 develop the operator versions of the Verblunsky recurrence relations, the Christoffel-Darboux formulas, and the Bernstein-Szeg\H{o} approximation. Section 7 presents the proof of the operator Szeg\H{o} limit theorem within the framework of matrix orthogonal polynomials on the unit circle.

\section{Preliminaries} 
In this section, we study the fundamental properties of positive operator-valued measures, their spectral representations, and connection to stationary processes.\\

Let $\mathcal{B}(\mathbb{T})$ denote the class of all Borel measurable subsets of the unit circle $\mathbb{T}$. We define $\mathcal{L}(\mathcal{H},\mathcal{K})$ as the algebra of all bounded linear operators from $\mathcal{H}$ to $\mathcal{K}$ where $\mathcal{H}$ and $\mathcal{K}$ are separable spaces. When $\mathcal{H} = \mathcal{K}$, we abbreviate this to $\mathcal{L}(\mathcal{H})$. Recall that $A \in \mathcal{L}(\mathcal{H})$ is positive definite $(A>0)$ if $\langle A x, x \rangle >0$ for $0 \neq x \in \mathcal{H}$, and $A$ has a bounded inverse.

\subsection{Operator-Valued Measures}

A function $\mu: \mathcal{B}(\mathbb{T}) \to \mathcal{L}(\mathcal{H})$ is called a \textit{positive operator-valued measure} (or \textit{semi-spectral measure}) on $\mathbb{T}$ if, for each $ h \in \mathcal{H}$, the scalar measure
\[
\mu_h(\sigma) = \langle \mu(\sigma) h, h \rangle,
\]
is a positive Borel measure on $\mathbb{T}$.

It can be shown that for each $ j \in \mathbb{Z}$, there exists an operator $ \mu(j) \in \mathcal{L}(\mathcal{H}) $, known as the \textit{$j$th moment} of $\mu$, such that for each $ h \in \mathcal{H}$,
\[
\langle \mu(j) h, h \rangle = \int_{\mathbb{T}} e^{i j t} d\mu_h(t).
\]

We assume that $ \mu $ is \textit{absolutely continuous} with respect to the Lebesgue measure on $\mathbb{T}$. That is, there exists a weakly measurable function $ \mu' : \mathbb{T} \to \mathcal{L}(\mathcal{H}) $, such that $ \mu'(e^{it}) \geq 0 $ almost everywhere on $\mathbb{T}$, and for all $ h \in \mathcal{H} $,
\[
 d \langle \mu(t) h, h \rangle = \langle \mu'(e^{it}) h, h \rangle dt.
\]
The $ j $th moment of $ \mu $ then corresponds to the Fourier coefficient $ \widehat{\mu'}(j) $ of $ \mu' $, leading to the expansion
\[
 \mu'(e^{it}) = \sum_{j=-\infty}^{\infty} \mu(j) e^{ijt}.
\]

\subsection{Spectral Representation and Stationary Processes}

Let $ X = (X_t)_{t \in G} $ be a stationary stochastic process indexed by a locally compact abelian group $ G $, satisfying the covariance condition:
\[
 C(t, s) = \mathbb{E}(X_s^* X_t) = C'(t - s).
\]
By the \textit{classical spectral theorem} \cite{DunSa}, there exists a unique unitary representation $ U : G \to \mathcal{L}(\mathcal{H}) $ satisfying
\[
 X_t = U_t X_0.
\]
The operator $ U_t $ is called the \textit{shift operator}, which captures the evolution of the process over time. By \textit{Stone’s theorem} \cite{Sto}, the shift operator has a spectral representation:
\[
 U = \int_{\mathbb{T}} e^{i\lambda} E(d\lambda),
\]
where $ E $ is an \textit{operator-valued spectral measure} on $\mathbb{T}$. Consequently, the process $ X_t $ admits the integral representation
\[
 X_t = U^t X_0 = \int_{\mathbb{T}} e^{it\lambda} E(d\lambda) X_0 = \int_{\mathbb{T}} e^{it\lambda} \xi(d\lambda),
\]
where $ \xi(\Delta) = E(\Delta) X_0 $ is a measure-valued process.

If $ G = \mathbb{Z} $, the covariance function satisfies
\[
 C_n = \langle X_{m+n}, X_m \rangle = \int_{\mathbb{T}} e^{-in\lambda} \left\langle \xi(d\lambda), \xi(d\lambda) \right\rangle.
\]
By uniqueness of the Fourier transform, the spectral measure $ \mu $ satisfies:
\[
 \mu(\Delta) = X_0^* E(\Delta) X_0.
\]

\subsection{Kolmogorov Isomorphism and Prediction Theory}

A key result in stochastic prediction is the \textit{Kolmogorov Isomorphism Theorem}, which asserts the existence of a unitary equivalence between the process $ X_n $ and a weighted shift operator on $ L^2(\mu, \mathcal{H}) $. The correspondence
\[
 X_n \leftrightarrow e^{-in\cdot} I, \quad n\in \mathbb{Z},
\]
where $ I $ is the identity operator on $ \mathcal{H} $, reformulates the Kolmogorov-Wiener prediction problem as an approximation problem in $ L^2(\mu, \mathcal{H}) $.\\

From (\cite{ManS} Theorem 4.19), the space $ L^2(\mu, \mathcal{H}) $ consists of operator-valued functions satisfying
\[
 \text{tr}  \int_{\mathbb{T}} f^\dagger(x) \mu(dx) f(x)  < \infty,
\]
equipped with the inner products
\[
 \left\langle \left\langle f, g \right\rangle \right\rangle_R = \int_{\mathbb{T}} f^\dagger(x) \mu(dx) g(x), \quad \left\langle \left\langle f, g \right\rangle \right\rangle_L = \int_{\mathbb{T}} g(x) \mu(dx) f^\dagger(x).
\]

\noindent An operator-valued measure is called non-trivial if $$\text{tr} \left\langle \left\langle f, g \right\rangle \right\rangle_R >0,$$ for every non-zero polynomial $f$. Here, \text{tr} is the trace.

\noindent Moreover, 
$$\left\langle \left\langle f, g \right\rangle \right\rangle_\mathrm{L} = \left\langle \left\langle g, f \right\rangle \right\rangle_\mathrm{L}^{\dag}, ~~~
\left\langle \left\langle f, g \right\rangle \right\rangle_\mathrm{R} = \left\langle \left\langle g, f \right\rangle \right\rangle_\mathrm{R}^{\dag}$$
$$\left\langle \left\langle f^{*}, g^{*} \right\rangle \right\rangle_\mathrm{L} = \left\langle \left\langle g, f \right\rangle \right\rangle_\mathrm{R}^{\dag}, ~~~
\left\langle \left\langle f^{*}, g^{*} \right\rangle \right\rangle_\mathrm{R} = \left\langle \left\langle g, f \right\rangle \right\rangle_\mathrm{L}^{\dag}$$
$$\left\langle \left\langle f, \alpha g \right\rangle \right\rangle_\mathrm{L} = \alpha \left\langle \left\langle f, g \right\rangle \right\rangle_\mathrm{L}^{\dag}, ~~~
\left\langle \left\langle \alpha f, g \right\rangle \right\rangle_\mathrm{L} = \left\langle \left\langle f, g \right\rangle \right\rangle_\mathrm{L} \alpha ^{\dag}$$
$$\left\langle \left\langle f, g \alpha \right\rangle \right\rangle_\mathrm{R} = \left\langle \left\langle f, g \right\rangle \right\rangle_\mathrm{R}\alpha, ~~~
\left\langle \left\langle f, \alpha g \right\rangle \right\rangle_\mathrm{R} = \alpha^{\dag} \left\langle \left\langle f, g \right\rangle \right\rangle_\mathrm{R}.$$

Consider a family of Hilbert spaces \( \mathbf{H} = \{ \mathcal{H}_n \}_{n \in \mathbb{Z}} \), where each \( \mathcal{H}_n \) is a separable Hilbert space. A function \( C: \mathbb{Z} \times \mathbb{Z} \to \mathcal{L}(\mathcal{H}_j, \mathcal{H}_i) \) is called a \textit{positive definite kernel} if it satisfies the condition
\[
\sum_{i,j} \langle C_{i,j} h_j, h_i \rangle \geq 0,
\]
for all finitely supported sequences \( \{ h_n \}_{n \in \mathbb{Z}} \) in the Hilbert direct sum \( \bigoplus_{n \in \mathbb{Z}} \mathcal{H}_n \). \\

The following theorem establishes a fundamental link between shift operators and positive definite sequences of operators. For further background, see \cite{Con}.

\begin{theorem} \label{thm:kolmoDecomp}[Kolmogorov decomposition]
Let \( C \) be a positive definite kernel. Then there exist a Hilbert space \( \mathcal{K} \) and a collection of bounded operators \( V_n \in \mathcal{L}(\mathcal{H}_n, \mathcal{K}) \) such that:
\begin{itemize}
    \item \( C_{i,j} = V_i^* V_j \), for all \( i,j \in \mathbb{Z} \);
    \item \( \mathcal{K} = \bigvee_{n \in \mathbb{Z}} V_n \mathcal{H}_n \), where \( \bigvee \) denotes the closed linear span.
\end{itemize}
\end{theorem}

A particularly important case occurs when the family \( \mathbf{H} \) reduces to a single Hilbert space, i.e., when \( \mathcal{H}_n = \mathcal{H} \) for all \( n \in \mathbb{Z} \), and the positive definite kernel \( C \) satisfies the property:
\[
C_{i,j} = T_{j-i},
\]
for some function \( T: \mathbb{Z} \to \mathcal{L}(\mathcal{H}) \). In this case, \( C \) is referred to as a \textit{positive definite Toeplitz kernel}.

\begin{theorem} \label{thm:naimark}[Naimark dilation]
Let \( C \) be a positive definite Toeplitz kernel. Then there exist a Hilbert space \( \mathcal{K} \), a unitary operator \( S \in \mathcal{L}(\mathcal{K}) \), and a bounded operator \( Q \in \mathcal{L}(\mathcal{H}, \mathcal{K}) \) such that:
\begin{itemize}
    \item \( C_{i,j} = Q^* S^{j-i} Q \), for all \( i,j \in \mathbb{Z} \);
    \item \( \mathcal{K} = \bigvee_{n \in \mathbb{Z}} S^n Q \mathcal{H}\).
\end{itemize}
\end{theorem}

The shift operator, which plays a fundamental role in functional and stochastic analysis, is defined by the successive powers of the unitary dilation operator \( U \). It is important to note that Theorem (\ref{thm:naimark}) applies only to \textit{stationary} processes. For nonstationary processes, a more general result is available: the Naimark dilation extends to the Kolmogorov decomposition, as described in Theorem \ref{thm:kolmoDecomp}.

This Kolmogorov decomposition extends the Kolmogorov Isomorphism Theorem (KIT) \cite{Kol} to the setting of operator-valued non-stationary processes, providing a unique minimal dilation representation for the correlation function \( C \) of a stochastic process. This fundamental result, originally established by Mandrekar and Salehi \cite{ManS} (\S 6), builds on the pioneering work of Wiener and Masani \cite{WieM} and reinforces the perspective proposed by Masani \cite{Mas2}, which asserts a deep connection between dilation theory and shift operators.

Dilation theory offers a powerful framework for analyzing signals by shifting the focus from the signal itself to its associated dilation operator, much as Fourier analysis shifts the emphasis from the time-domain representation to the spectral domain. Unlike Fourier methods, which are generally limited to stationary processes, dilation theory does not impose stationarity constraints. 

In the case of stationary processes, dilation theory naturally leads to the Naimark dilation (Theorem \ref{thm:naimark}), providing a unitary dilation framework. However, for non-stationary processes, the appropriate generalization is given by the Kolmogorov decomposition (Theorem \ref{thm:kolmoDecomp}), which constructs a minimal isometric dilation that captures the underlying structure of the non-stationary process.

\section{Operator orthogonal polynomials}

In Szeg\H{o} theory, orthogonal polynomials play a prominent role. Here, we will need to work with operator-valued orthogonal polynomials with respect to operator-valued measures on the unit circle. \\

For an operator polynomial $P_{n}$ of degree $n$, we defined the \textit{reversed polynomial} $P_{n}^{*}$  of $P_{n}$ by $$
P^{*}_{n}(z)=z^{n}P_{n}(1/\bar{z})^{\dag}.
$$
\noindent We have $$(P^{*}_{n})^{*}= P_{n},$$ and for any $\alpha\in \mathcal{L}(\mathcal{H})$, $$(\alpha P_{n})^{*} = P_{n}^{*}\alpha^{\dag},~~(P_{n}\alpha )^{*} = \alpha^{\dag}P_{n}^{*}.$$

The following lemma will be a very useful characterization of positive definite 2 $\times$ 2 operator matrices.

\begin{lem}
The following are equivalent:
\begin{enumerate}
% [label=\roman*]
  \item The operator matrix 
  	   $
       \begin{pmatrix}
       A & B\\
       B^{*} & C
       \end{pmatrix}
       $
       is positive definite. \\
  \item $A>0$ and $C-B^{*}A^{-1}B>0$.
  \item $C>0$ and $A-B^{*}C^{-1}B>0$;
\end{enumerate}
here $C-B^{*}A^{-1}B>0$ is called the {\it Schur complement} of $A$.
\end{lem}

\begin{proof}
This follows immediately from the \it{Frobenius-Schur factorization} (\cite{Zha}), 
\begin{equation}\label{schurcomplement}
\begin{pmatrix}
A & B\\
B^{*} & C
\end{pmatrix}
=\begin{pmatrix}
I & 0\\
B^{*}A^{-1} & I
\end{pmatrix}
\begin{pmatrix}
A & 0\\
0 & C-B^{*}A^{-1}B
\end{pmatrix}
\begin{pmatrix}
I & A^{-1}B\\
0 & I
\end{pmatrix}.
\end{equation}
\end{proof}

Given an operator-valued positive measure on $\mathbb{T}$, we define its \textit{moments} for $j=-n, \cdots, n$ by $$\mu_{j}=\mu_{j}(d\mu)=\frac{1}{2\pi}\int_{-\pi}^{\pi} e^{-ij\theta} \mu(d\theta) ~~ \mathrm{and}~~ \mu_{-j}=\mu^{*}_{j}.$$
The right and left {\it Toeplitz operator matrices} $T^{R}_{n}$ and $T^{L}_{n}$ associated to $\mu$ are
\[ \label{toeplitz_TR}
T^{R}_{n}=
\begin{pmatrix}
\mu_{0} & \mu_{-1} & \cdots & \mu_{-n+1}\\
\mu_{1} & \mu_{0} & \cdots & \mu_{-n+2}\\
\vdots & \vdots & \ddots & \vdots\\
\mu_{n-1} & \mu_{n-2} & \cdots & \mu_{0}
\end{pmatrix},
%&
T^{L}_{n}=
\begin{pmatrix}
\mu_{0} & \mu_{1} & \cdots & \mu_{n-1}\\
\mu_{-1} & \mu_{0} & \cdots & \mu_{n-2}\\
\vdots & \vdots & \ddots & \vdots\\
\mu_{-n+1} & \mu_{-n+2} & \cdots & \mu_{0}
\end{pmatrix}.
\]

\noindent We also have 
\begin{align*}
T^{R}_{n+1}&=
\begin{pmatrix}
T^{R}_{n} & \nu_{n}\\
\nu^{*}_{n} & \mu_{0}
\end{pmatrix},
&
T^{L}_{n+1}&=
\begin{pmatrix}
T^{L}_{n} & \xi_{n}\\
\xi^{*}_{n} & \mu_{0}
\end{pmatrix},
\end{align*}

\noindent where  
\begin{align*}
\nu_{n}&=
\begin{pmatrix}
\mu_{-n}\\
\mu_{-n+1}\\
\vdots \\
\mu_{-1}
\end{pmatrix},
&
\xi_{n}&=
\begin{pmatrix}
\mu_{n}\\
\mu_{n-1}\\
\vdots \\
\mu_{1}
\end{pmatrix},
\end{align*}

\noindent and we define the {\it Schur complements} $\kappa_{n}^{R,L}$ of $\mu_{0}$ {\it in} $T^{R,L}_{n+1}$ as
\begin{eqnarray*}
\kappa_{n}^{R} &=& \mathrm{SC}(T^{R}_{n+1}) = \mu_{0} - \nu^{*}_{n} T^{-R}_{n}\nu_{n}, \\
\kappa_{n}^{L} &=& \mathrm{SC}(T^{L}_{n+1}) = \mu_{0} - \xi^{*}_{n} T^{-L}_{n}\xi_{n},
\end{eqnarray*}
\noindent where  $T^{-R,-L}_{n}= (T^{R,L}_{n})^{-1}$.

\begin{remark} 
If $P(z)=\sum_{k=0}^{n}p_k z^k$ and $Q(z)=\sum_{k=0}^{n}q_k z^k$, then $$\left\langle \left\langle P, Q \right\rangle \right\rangle_\mathrm{R}= \sum_{k,j=0}^{n}p_k^{\dag} T_{k-j}^{R}q_j,\quad \left\langle \left\langle P, Q \right\rangle \right\rangle_\mathrm{L}= \sum_{k,j=0}^{n}q_j T_{k-j}^{L}p_j^{\dag}.$$
This shows that the positivity of $\mu$ implies the positivity of $T^{R,L}_{n}$ and $\kappa_{n}^{R,L}$, hence they have bounded inverses. Indeed, we have 
\begin{eqnarray*}
0 &<& \int_{\mathbb{T}}\left\langle \left\langle d\mu(\sigma) \sum_{j=0}^{n}p_j z^j, \sum_{k=0}^{n}p_k z^k \right\rangle \right\rangle_\mathrm{R,L} \\
  &<& \sum_{j,k=0}^{n}\int_{\mathbb{T}}z^{j-k}\left\langle \left\langle d\mu(\sigma) p_j, p_k \right\rangle \right\rangle_\mathrm{R,L}=\sum_{j,k=0}^{n}\left\langle \left\langle T_{k-j} p_j, p_k \right\rangle \right\rangle_\mathrm{R,L},
\end{eqnarray*}
so $T^{R,L}_{n}$ is positive definite, and has a bounded inverse $T^{-R,-L}_{n}$.  
\end{remark} 

\section{Verblunsky recursion}

In this section, we show that the right and left orthogonal polynomials follow the classical Verblunsky recurrence relations.\\
\noindent In our approach, we will define the monic operator-valued polynomials in terms of the Schur complements. The notion of Schur complements (SC) is relatively simple but suprisingly strong.

\begin{prop}
Monic polynomials such that 

\begin{eqnarray*}
\Phi_{n}^{R}(z)&=&\mathrm{SC} 
\begin{pmatrix}
\mu_{0} & \mu_{-1} & \cdots & \mu_{-n+1} & \mu_{-n}\\
\mu_{1} & \mu_{0} & \cdots & \mu_{-n+2} & \mu_{-n+1}\\
\vdots & \vdots & \ddots & \vdots & \vdots\\
\mu_{n-1} & \mu_{n-2} & \cdots & \mu_{0}  & \mu_{-1}\\
I & zI & \cdots & z^{n-1}I & z^{n}I
\end{pmatrix} \\
&=& z^{n}I-[I~~zI~~ \cdots ~~ z^{n-1}I]~T^{-R}_{n}
\begin{pmatrix}
\mu_{-n}\\
\mu_{-n+1}\\
\vdots \\
\mu_{-1}
\end{pmatrix},
\end{eqnarray*}

\begin{eqnarray*}
\Phi_{n}^{L}(z)&=&\mathrm{SC} 
\begin{pmatrix}
\mu_{0} & \mu_{-1} & \cdots & \mu_{n-1} & I\\
\mu_{1} & \mu_{0} & \cdots & \mu_{n-2} & zI\\
\vdots & \vdots & \ddots & \vdots & \vdots\\
\mu_{n-1} & \mu_{n-2} & \cdots & \mu_{0}  & z^{n-1}I\\
\mu_{-n} & \mu_{n-1} & \cdots & \mu_{-1} & z^{n}I
\end{pmatrix} \\
&=& z^{n}I-[\mu_{-n} ~~ \mu_{-n+1} ~~ \cdots ~~ \mu_{-1}]~T^{-L}_{n}
\begin{pmatrix}
I\\
zI\\
\vdots \\
z^{n-1}I
\end{pmatrix}
\end{eqnarray*}
are orthogonal, i.e., for any $k,j \geq 0$,
$$
\left\langle \left\langle \Phi_{k}^{R},\Phi_{j}^{R}\right\rangle\right\rangle_{\mathrm{R}} = \delta_{kj}\kappa_{k}^{R},  \qquad
\left\langle \left\langle \Phi_{k}^{L},\Phi_{j}^{L}\right\rangle\right\rangle_{\mathrm{L}} =\delta_{kj}\kappa_{k}^{L}.
$$

\end{prop}
\begin{proof}
We prove just the first relation; the second is similar.\\
\noindent For any $0 \leq m \leq n-1$ and $z=e^{i\theta}$
\begin{eqnarray*}
\left\langle \left\langle z^{m}I,\Phi_{n}^{R}\right\rangle\right\rangle_{\mathrm{R}} &=& 
\int_{-\pi}^{\pi}e^{-im\theta} \mu(\theta) 
(e^{i n \theta} - 
[I~~e^{i\theta}~~\cdots~~e^{i(n-1)\theta}]T^{-R}_{n}\nu_{n})d\theta\\
&=& \mu_{m-n} - [\mu_{m}~~\cdots~~\mu_{m-n+1}]T^{-R}_{n}\nu_{n}\\
&=& \mu_{m-n} - \mu_{m-n} = 0.
\end{eqnarray*}
If $m=n$, then 
\begin{eqnarray*}
\left\langle \left\langle \Phi_{n}^{R},\Phi_{n}^{R}\right\rangle\right\rangle_{\mathrm{R}} &=& 
\mu_{0} - [\mu_{n}~~\cdots~~\mu_{1}]T^{-R}_{n}\nu_{n} = \kappa_{n}^{R}.
\end{eqnarray*}
\end{proof}

\begin{theorem}
The monic orthogonal polynomials $\Phi^{R}_{n}$ and $\Phi^{L}_{n}$ obey the following recursion relations:
\begin{eqnarray} \label{monic_recursion}
\Phi_{n+1}^{R} &=& z \Phi_{n}^{R} + \Phi_{n}^{L,*} \Phi_{n+1}^{R}(0), \\
\Phi_{n+1}^{L} &=& z \Phi_{n}^{L} + \Phi_{n+1}^{L}(0) \Phi_{n}^{R,*}.
\end{eqnarray} 
\end{theorem}     
\begin{proof}
For the first recursion, observe that for any $1 \leq i \leq n$, we have 
$$\left\langle \left\langle \Phi_{n+1}^{R} - z \Phi_{n}^{R}, z^{i}I\right\rangle\right\rangle_{\mathrm{R}}= \left\langle \left\langle \Phi_{n}^{L,*}, z^{i}I\right\rangle\right\rangle_{\mathrm{R}} = 0,$$
and so $\Phi_{n+1}^{R} - z \Phi_{n}^{R}$ and $\Phi_{n}^{L,*}$ are proportional. Setting $z=0$ gives the constant of proportionality as $\Phi_{n+1}^{R}(0)$. 
 Similarly for the other claim.
 \end{proof} 
 
\noindent Based on Miamee and Salehi \cite{MiaS1}, the operator $\kappa^{R,L}_{n}$  has a well-defined square root that is positive definite. Therefore, we define the normalized orthogonal polynomials as follows:
$$
\varphi_{n}^{R} =\Phi_{n}^{R} \kappa^{-R/2}_{n}~~~~\mathrm{and}~~~~\varphi_{n}^{L} =\kappa^{-L/2}_{n} \Phi_{n}^{L}.
$$
\noindent One easily verifies 
$$
\left\langle \left\langle \varphi_{n}^{R,L},\varphi_{n}^{R,L} \right\rangle\right\rangle_{\mathrm{R,L}}=\kappa^{-R,L/2}_{n}\left\langle \left\langle \Phi_{n}^{R,L},\Phi_{n}^{R,L} \right\rangle\right\rangle_{\mathrm{R,L}}\kappa^{-R,L/2}_{n}=I.
$$

\noindent Defining 
$$
\rho_{n}^{R}=\kappa_{n+1}^{R/2}\kappa_{n}^{-R/2} ~~\mathrm{and}~~ \rho_{n}^{L}=\kappa_{n}^{-L/2}\kappa_{n+1}^{L/2},
$$
\noindent one can easily show 
\begin{eqnarray}\label{szego-recurrence}
z \varphi_{n}^{R} - \varphi_{n+1}^{R} \rho^{R}_{n}&=&  \varphi_{n}^{L,*}(\alpha_{n}^{R})^{\dag},  \label{szego-recurrence_right} \\
z \varphi_{n}^{L} - \rho^{L}_{n} \varphi_{n+1}^{L}&=&  (\alpha_{n}^{L})^{\dag} \varphi_{n}^{R,*}, \label{szego-recurrence_left}
\end{eqnarray} 
\noindent where 
\begin{eqnarray}
\alpha_{n}^{R} &=& -(\kappa^{-R/2}_{n})^{\dag} \Phi_{n+1}^{R}(0)^{\dag} \kappa^{L/2}_{n},  \label{eqn:alph_R}\\
\alpha_{n}^{L} &=& -\kappa^{R/2}_{n} \Phi_{n+1}^{L}(0)^{\dag} (\kappa^{-L/2}_{n})^{\dag}.
\end{eqnarray}

\noindent One also has 
$$
\kappa^{-L}_{n}=(\rho_{n-1}^{L} \dots \rho_{0}^{L})^{2}~~\mathrm{and}~~\kappa^{-R}_{n}=(\rho_{0}^{R} \dots \rho_{n-1}^{R})^{2}.
$$

\begin{prop}\label{prop4.2}
\hfill
\begin{enumerate}
    \item The operators $\alpha_{n}^{R}$ and $\alpha_{n}^{L}$ are equal, denoted $\alpha_n$.
    \item $\rho^{L}_{n} = (I - \alpha_{n}^{\dag}\alpha_{n})^{1/2}$ and $\rho^{R}_{n} = (I - \alpha_{n}\alpha_{n}^{\dag})^{1/2}$.
\end{enumerate}
\end{prop}

\begin{proof}
\noindent 
1) Observing that $\left\langle \left\langle \varphi_{n}^{R},\varphi_{n}^{L,*}\right\rangle\right\rangle_{\mathrm{R}} = \left\langle \left\langle \varphi_{n}^{R,*},\varphi_{n}^{L} \right\rangle\right\rangle_{\mathrm{L}} $, and using equations (\ref{szego-recurrence}), we can easily prove that $\alpha_{n}^{R}=\alpha_{n}^{L}$.

2) From (\ref{szego-recurrence_left}), we have 

\begin{eqnarray*}
||z\varphi_{n}^{L}||_{\mathrm{L}}- (\rho^{L}_{n})^{2} &=& \alpha_{n}^{\dag} \left\langle \left\langle \varphi_{n}^{R,*},\varphi_{n}^{R,*} \right\rangle\right\rangle_{\mathrm{L}} \alpha_{n}\\
&=& \alpha_{n}^{\dag}\alpha_{n}.
\end{eqnarray*}
Since $||z\varphi_{n}^{L}||_{\mathrm{L}} =1$, we have  $(\rho^{L}_{n})^{2} = 1 - \alpha_{n}^{\dag}\alpha_{n}.$\\
For $\rho^{R}_{n}$, the proof is similar.
\end{proof}

\noindent Using Proposition \ref{prop4.2}, we write,  
\begin{equation} \label{eq:recurrence_matrix}
      \begin{pmatrix}
        \varphi_{n}^{L}(z) \\
        \varphi_{n}^{R,*}(z) 
        \end{pmatrix}= A^{L}(\alpha_{n},z)\begin{pmatrix}
        \varphi_{n}^{L}(z) \\
        \varphi_{n}^{R,*}(z) 
      \end{pmatrix}
\end{equation}

where
\begin{equation} \label{eq:recurrence_inv_matrix}
A^{L}(\alpha,z) = \begin{pmatrix}
z(\rho^{L})^{-1} & -(\rho^{L})^{-1} \alpha^{\dag}\\
-z(\rho^{R})^{-1} \alpha & (\rho^{R})^{-1}
\end{pmatrix}.
\end{equation}

\section{Christoffel-Darboux Formulae}

This section introduces operator-valued right and left kernel polynomials and uses the recurrence formulae to derive the Christoffel-Darboux formulae. 

\begin{prop}\label{prop_kernelCD}
Define the right and left kernel polynomials of degree $n$ as
$$K_{n}^{R}(w,z) = \sum_{k=0}^{n}\varphi_{k}^{R}(z)\varphi_{k}^{R}(w)^{\dag} ,\quad K_{n}^{L}(w,z) = \sum_{k=0}^{n}\varphi_{k}^{L}(z)^{\dag}\varphi_{k}^{L}(w),$$
Then, we have
\begin{enumerate}
%[label=(\roman*)]
\item
$
K_{n}^{R}(w,z)=[I ~ zI ~\cdots ~ z^{n}I]~T_{n+1}^{-R}~\begin{bmatrix}
I \\
w^{-1}I \\
\vdots \\
w^{-n}I
\end{bmatrix}.
$
\item 
$
K_{n}^{L}(w,z)=[I ~ zI ~\cdots ~ z^{-n}I]~T_{n+1}^{-L}~\begin{bmatrix}
I \\
w^{1}I \\
\vdots \\
w^{n}I
\end{bmatrix}.
$
\item \begin{align} \label{cd_right}
K_{n}^{L}(w,z) &=&\frac{\varphi_{n}^{R,*}(w)^{\dag}\varphi_{n}^{R,*}(z) - \bar{w}z\varphi_{n}^{L}(w)^{\dag}\varphi_{n}^{L}(z)}{1-\bar{w}z}, \\ 
\label{cd_left}
K_{n}^{R}(w,z) &=&\frac{\varphi_{n}^{L,*}(w)\varphi_{n}^{L,*}(z)^{\dag} - \bar{w}z\varphi_{n}^{R}(w)\varphi_{n}^{R}(z)^{\dag}}{1-\bar{w}z}. 
\end{align}
\end{enumerate}
\end{prop}

\begin{proof}
We write $$
K_{n}^{R}(w,z)=[Z ~ z^{n}I]~T_{n+1}^{-R}~\begin{bmatrix}
W^{*}\\
w^{-n}I
\end{bmatrix},
$$ with $Z=[I ~ zI ~\cdots ~ z^{n-1}I]$ and $W=[I ~ wI ~\cdots ~ w^{n-1}I]$.\\
Writing  
$$
T^{R}_{n+1}=
\begin{pmatrix}
T^{R}_{n} & \nu_{n}\\
\nu^{*}_{n} & \mu_{0}
\end{pmatrix}
$$ 
and using (\ref{schurcomplement}), we have 
$$
T^{-R}_{n+1}=
\begin{pmatrix}
A & \gamma\\
\gamma^{*} & \alpha
\end{pmatrix},
$$
where 
$$
A = T^{-R}_{n} + T^{-R}_{n}\nu_{n}\kappa_{n}^{-R}\nu_{n}^{*}T^{-R}_{n},~~~ \gamma = -T^{-R}_{n}\nu_{n}\kappa_{n}^{-R},~~~ \alpha = \kappa_{n}^{-R}.
$$
Using the equality 
$$
[I ~ zI ~\cdots ~ z^{n-1}I]T^{-R}_{n}\nu_{n}\kappa_{n}^{-R/2}
= z^{n}\kappa_{n}^{-R/2} - \varphi_{n}^{R}(z),
$$
we rewrite $K_{n}^{R}(w,z)$ as follows:
\begin{align*}
K_{n}^{R}(w,z)
&= ZAW^{*} + z^{n}\gamma^{*}X^{*}+Z\gamma w^{-n} + z^{n}\alpha x^{-n}\\
&= Z \Bigl(T_{n}^{-R} \;+\; T_{n}^{-R}\,\nu_{n}\,\kappa_{n}^{-R}\,\nu_{n}^{*}\,T_{n}^{-R}\Bigr) W^{*}
   \;-\; z^{n}\,\kappa_{n}^{-R}\,\nu_{n}^{*}\,T_{n}^{-R}\\
   \;& \quad -\; Z\,T_{n}^{-R}\,\nu_{n}\,\kappa_{n}^{-R}\,w^{-n}   
   \;+\; z^{n}\,\kappa_{n}^{-R} w^{-n}\, \\[6pt]
&= Z\,T_{n}^{-R}\,W^{*}
   \;+\;\bigl(z^{n}\,\kappa_{n}^{-R/2}\;-\;\varphi_{n}^{R}(z))\,
    \bigl(w^{-n}\,\kappa_{n}^{-R/2}\;-\;\varphi_{n}^R(w)^{\dag}) \\
&\quad -\;w^{-n}\,\bigl(z^{n}\,\kappa_{n}^{-R/2}\;-\;\varphi_{n}^{R}(z)\bigr)\,\kappa_{n}^{-R/2}
   \;-\;z^{n}\,\kappa_{n}^{-R/2}\,\bigl(w^{-n}\,\kappa_{n}^{-R/2}\;-\;\varphi_{n}^{R}(w)^{\dag}) \\
&\quad +\;z^{n}\,w^{-n}\,\kappa_{n}^{-R} \\[6pt]
&= K_{n-1}^{R}(w,z) + \varphi_{n}^{R}(z) \varphi_{n}^{R}(w)^{\dag}.
\end{align*}
Summing, the telescoping sum gives 1. as required.\\

\noindent For 3.
$$
F_{n}^{L}(z)=\begin{pmatrix}
\varphi_{n}^{L}(z) \\
\varphi_{n}^{R,*}(z) 
\end{pmatrix}, ~~ 
J=\begin{pmatrix}
I & 0 \\
0 & -I
\end{pmatrix}, ~~
\tilde{J}=\begin{pmatrix}
\bar{w}zI & 0 \\
0 & -I
\end{pmatrix}.
$$
Then 
$$
F_{n+1}^{L}(z)= A^{L}(\alpha_{n},z)F_{n}^{L}(z)
$$ 
and 
$$
A^{L}(\alpha_{n},w)^{\dag} J A^{L}(\alpha_{n},z)=\begin{pmatrix}
\bar{w}zI & 0 \\
0 & -I
\end{pmatrix} = \tilde{J}.
$$ 
Thus 
$$
F_{n+1}^{L}(w)^{\dag} J F_{n+1}^{L}(z) = F_{n}^{L}(w)^{\dag} A^{L}(\alpha_{n},w)^{\dag} J A^{L}(\alpha_{n},z) F_{n}^{L}(z) = F_{n}^{L}(w)^{\dag} \tilde{J} F_{n}^{L}(z)
$$ 
and hence
$$
\varphi_{n+1}^{L}(w)^{\dag}\varphi_{n+1}^{L}(z) - \varphi_{n+1}^{R,*}(w)^{\dag}\varphi_{n+1}^{R,*}(z)=\bar{w}z\varphi_{n}^{L}(w)^{\dag}\varphi_{n}^{L}(z) - \varphi_{n}^{R,*}(w)^{\dag}\varphi_{n}^{R,*}(z),
$$ 
which gives the second part of (a). Now, if we let 
$$
Q_{n}^{L}(z,w)=\varphi_{n+1}^{R,*}(w)^{\dag}\varphi_{n+1}^{R,*}(z) - \varphi_{n+1}^{L}(z)^{\dag}\varphi_{n+1}^{L}(z),
$$ 
we see 
\begin{eqnarray*}
Q_{n+1}^{L}(z,w) - Q_{n}^{L}(z,w) &=& \varphi_{n}^{R,*}(w)^{\dag}\varphi_{n}^{R,*}(z) - \bar{w}z\varphi_{n}^{L}(w)^{\dag}\varphi_{n}^{L}(z) \\
& &- \varphi_{n}^{R,*}(w)^{\dag}\varphi_{n}^{R,*}(z) + \varphi_{n}^{L}(w)^{\dag}\varphi_{n}^{L}(z)\\
&=&(1-\bar{w}z)\varphi_{n}^{L}(w)^{\dag}\varphi_{n}^{L}(z).
\end{eqnarray*}
Summing over $n$ the proof for (\ref{cd_right}) follows since $Q_{-1}^{L}(z,w)=0$.\\
The proof for (\ref{cd_left}) is similar.
\end{proof}

\section{Bernstein-Szeg\H{o} Approximation}
In this section, we will reconstruct the operator-valued measure from the operator orthogonal polynomials via the Bernstein–Szeg\H{o} approximation.\\

\noindent Given a nontrivial operator-valued measure $d\mu$, with Verblunsky coefficients $\{\alpha_{n}\}_{n=0}^{\infty}$, we will define measures $d\mu^{(n)}$ with
\[
\alpha_{j}(d\mu^{(n)})= \left\{ 
	\begin{array}{ll}
        \alpha_{j},& \mbox{ $j\leq n$}  \\
        \mathbf{0},  & \mbox{ $j\geq n+1$}\\
    \end{array}
    \right.
\]
\noindent then $d\mu^{(n)} \rightarrow d\mu$ weakly (as in Simon \cite{Sim2} in the scalar case). As \cite{Sim2}, the proof below of the Szeg\H{o} theorem hinges on this.\\

\noindent As a preliminary, we need the following theorem:

\begin{theorem} \label{bernstein_zeros} We have
\begin{enumerate}
% [label=(\roman*)]
\item For $z \in \mathbb{T}$, all of $\varphi_{n}^{R,*}(z)$, $\varphi_{n}^{L,*}(z)$, $\varphi_{n}^{R}(z)$, $\varphi_{n}^{L}(z)$ are invertible.
\item For any $z \in \mathbb{T}$, $\varphi_{n}^{R}(z)\varphi_{n}^{R}(z)^{\dag} = \varphi_{n}^{L}(z)^{\dag}\varphi_{n}^{L}(z).$
\end{enumerate}

\end{theorem}
\begin{proof}
For 1., assume that $0< |w| \leq 1$. If there exists $c \neq 0$ such that $\varphi_{n}^{R}(w)^{\dag}c=0$, we also have that $\varphi_{n}^{L,*}(w)^{\dag}c=0$.\\
\noindent Then, propositions \ref{prop_kernelCD} lead to
$$
[I ~ zI ~\cdots ~ z^{n}I]~T_{n+1}^{-R}~\begin{bmatrix}
c \\
w^{-1}c \\
\vdots \\
w^{-n}c
\end{bmatrix}=0,~~ \mathrm{for}~~z \in \mathbb{C},
$$
which contradicts the invertibility of $T_{n+1}^{-R}$, unless $c=0$. This proves 1. for $\varphi_{n}^{L,*}(z)$ and $\varphi_{n}^{R}(z)$. If $z = e^{i\theta}$, $\varphi_{n}^{R,*}(z)= e^{-in\theta} \varphi_{n}^{R}(z)^{\dag}$ and $\varphi_{n}^{L}(z)= e^{in\theta} \varphi_{n}^{L,*}(z)^{\dag}$ are also invertible. \\
\indent For 2.,  put $z=w$ in equation (\ref{cd_right}).  One has $\varphi_{n}^{R,*}(z)^{\dag}\varphi_{n}^{R,*}(z) = \varphi_{n}^{L}(z)^{\dag}\varphi_{n}^{L}(z)$, which can be rewritten as 
$$
\varphi_{n}^{R}(z)\varphi_{n}^{R}(z)^{\dag} = \varphi_{n}^{L}(z)^{\dag}\varphi_{n}^{L}(z),
$$
as required.
% \indent For 2., Let $\varphi_{n}^{R}(z_{0}) = 0$ and define $p_{n-1}$ such that $(z-z_{0})p_{n-1}= \varphi_{n}^{R}$. The polynomial $p_{n-1}$ is of degree $n-1$, which implies $\left\langle \left\langle p_{n-1},\varphi_{n}^{R} \right\rangle\right\rangle_{\mathrm{R}} = 0$. So 
% $$
% \left\langle \left\langle zp_{n-1}, zp_{n-1}\right\rangle\right\rangle_{\mathrm{R}}
% =  \left\langle \left\langle p_{n-1}, p_{n-1}\right\rangle\right\rangle_{\mathrm{R}}
% = |z_{0}|^{2}\left\langle \left\langle p_{n-1}, p_{n-1}\right\rangle\right\rangle_{\mathrm{R}}+ 
% \left\langle \left\langle \varphi_{n}^{R}, \varphi_{n}^{R}\right\rangle\right\rangle_{\mathrm{R}},$$ or $$(1-|z_{0}|^{2})\left\langle \left\langle p_{n-1}, p_{n-1}\right\rangle\right\rangle_{\mathrm{R}} = \left\langle \left\langle \varphi_{n}^{R}, \varphi_{n}^{R}\right\rangle\right\rangle_{\mathrm{R}},
% $$ 
% from which we conclude $|z_0| < 1$, that is, the zeros of $\varphi_{n}$ lie in $\mathbb{D}$. Since $\varphi_{n}^{R,*}(z_{0})=0$ if and only if $\varphi_{n}^{R}(1/\bar{z}_{0})=0$, the zeros of $\varphi_{n}^{R,*}$ lie in $\mathbb{C}\backslash \mathbb{\bar{D}}$.
\end{proof}

As in the scalar case, we can use the invertibility in Theorem (\ref{bernstein_zeros}) to write 
\begin{equation} \label{eq:BS_measure}
d\mu^{(n)}(\theta) = [\varphi_{n}^{R}(e^{i\theta})\varphi_{n}^{R}(e^{i\theta})^{\dag}]^{-1}\frac{d\theta}{2\pi} = \mu'^{(n)}(\theta)\frac{d\theta}{2\pi}.
\end{equation} 
Also, directly from the definition of the right orthogonal polynomials, we have the right Bernstein-Szeg\H{o} approximation to $\mu$:
\[
d\mu^{(n)}(\theta) = [\varphi_{n}^{R,*}(e^{i\theta})^{\dag}\varphi_{n}^{R,*}(e^{i\theta})]^{-1}\frac{d\theta}{2\pi};
\]
\noindent the corresponding left Berstein-Szeg\H{o} approximation of ${\mu}_n     $ to $\mu$ is 
\[
d\mu^{(n)}(\theta) = [\varphi_{n}^{L,*}(e^{i\theta})\varphi_{n}^{L,*}(e^{i\theta})^{\dag}]^{-1}\frac{d\theta}{2\pi}.
\]

\begin{theorem}
The operator-valued measure $d\mu^{(n)}$ is normalized and its right operator orthogonal polynomials for $j=0, \cdots , n$ are $\{\varphi_{j}^{R}\}_{j=0}^{n}$, and for $j>n$, 
\begin{equation}\label{eq1_bernstein}
    \varphi_{j}^{R}(z; d\mu^{(n)})=z^{j-n}\varphi_{n}^{R}(z; d\mu).
\end{equation}

\noindent The Verblunsky coefficients for $d\mu^{(n)}$ are 
\begin{eqnarray}\label{eq:BD_approx_bis}
\alpha_{j}(d\mu^{(n)})= \left\{ 
	\begin{array}{ll}
        \alpha_{j}(d\mu),& \mbox{ $j\leq n$} , \\
        \mathbf{0},  & \mbox{ $j\geq n+1$}.\\
    \end{array}
    \right.    
\end{eqnarray} 
\end{theorem}
\begin{proof}
Let $\langle\langle \cdot, \cdot\rangle \rangle_{R}$ be the inner product associated with $\mu^{(n)}$. By direct calculation, 
\begin{equation}
\langle\langle \varphi_{n}^{R}, \varphi_{n}^{R} \rangle \rangle_{R} = \mathbf{1},
\end{equation}
and for $j = 0,1, \cdots, n-1$,
\begin{equation}
\langle\langle \varphi_{n}^{R}, \varphi_{j}^{R} \rangle \rangle_{R} = 0.
\end{equation}
So the family $\{\varphi_{n}^{R}\}_{j=0}^{n}$ is an orthonormal basis with respect to the inner product $\langle\langle \cdot, \cdot\rangle \rangle_{R}$. Also,
\begin{eqnarray*}
    \langle\langle e^{ij\theta}, \varphi_{n}^{R} \rangle \rangle_{R} &=&\frac{1}{2\pi}\int_{0}^{2\pi}e^{ij\theta}(\varphi_{n}^{R}(e^{i\theta})^{\dag})^{-1}d\theta\\
    &=& \frac{1}{2\pi}\oint e^{i(n-j-1)\theta}(\varphi_{n}^{R, *}(e^{i\theta}))^{-1}d\theta = 0,
\end{eqnarray*}
\noindent since $n-k-1 \geq 0$ and $\varphi_{n}^{R, *}(e^{i\theta})^{-1}$ is analytic in a neighborhood of $\bar{\mathbb{D}}$ by Theorem (\ref{bernstein_zeros}). This proves $\varphi_{n}^{R}$ is an OPUC for $d\mu_{n}$ and (\ref{eq1_bernstein}) holds. \\
By (\ref{eqn:alph_R}), if $\varphi^{R}_{k+1}(0)=0$, then $\alpha_k=0$, then by (\ref{eq1_bernstein}) $\varphi_{n+j}^{R}(0; d\mu^{(n)})=0$, which implies (\ref{eq:BD_approx_bis}).
\end{proof}
\noindent Next, measures with the properties above are essentially unique:
\begin{theorem} \label{geronimus_thm}
    Let $d\mu$ and $d\nu$ be two nontrivial operator-valued measures on $\mathbb{T}$ such that for $n$, 
    \begin{equation}\label{eq:Phi_N}
           \varphi_{n}^{R,L}(z;d\mu) = \varphi_{n}^{R,L}(z;d\nu).
    \end{equation} 
Then 
    \begin{eqnarray} \label{eq:Geronimus_op}
    	\begin{array}{ll}
            (i)~~\varphi^{R,L}_{j}(z;d\mu) = \varphi_{j}^{R,L}(z;d\nu)& \mbox{   $j=0,1,\cdots,n-1$} , \\
            (ii)~~\alpha_{j}(d\mu) = \alpha_{j}(d\nu)  & \mbox{   $j=0,1,\cdots,n-1$},\\
            (iii)~~\mu_{j}(d\mu) = \mu_{j}(d\nu)  & \mbox{  $j=0,1,\cdots,n$}.\\
            \end{array}  
    \end{eqnarray} 
\end{theorem}

\begin{proof}
    The recurrence (\ref{monic_recursion}) can be written in the matrix form
    $$
    \begin{pmatrix}
    \varphi_{j+1}^{R}(z) \\
    \varphi_{j+1}^{L,*}(z) 
    \end{pmatrix}= A^{R}(\alpha_{j},z)\begin{pmatrix}
    \varphi_{j}^{R}(z) \\
    \varphi_{j}^{L,*}(z) 
    \end{pmatrix},
    $$ 
    where
    $$
    A^{R}(\alpha,z) = \begin{pmatrix}
    z \rho^{-R} & -z \alpha \rho^{-L} \\
    - \alpha^{\dag} \rho^{-R}  & \rho^{-L}
    \end{pmatrix},
    $$
    and its inverse for $z \neq 0$,
    $$
    A^{-R}(\alpha,z) = \begin{pmatrix}
    z^{-1}  \rho^{-R} & \rho^{-R}\alpha\\
    z^{-1} \rho^{-L} \alpha^{\dag} & \rho^{-L}
    \end{pmatrix},
    $$
    which gives the inverse Szeg\H{o} relations 
    \begin{eqnarray} 
    \varphi_{j}^{R}(z) &=& z^{-1} \varphi_{j+1}^{R}(z) \rho_{j}^{-R} + z^{-1} \varphi_{j+1}^{L,*}(z) \rho_{j}^{-L} \alpha^{\dag}_j \label{monic_recursion_inv_1},\\
    \varphi_{j}^{L,*}(z) &=& \varphi_{j+1}^{R}(z) \rho_{j}^{-R} \alpha^{\dag}_j + \varphi_{j+1}^{L,*}(z) \rho_{j}^{-L}. \label{monic_recursion_inv_2}
    \end{eqnarray} 
    Then (\ref{eq:Phi_N}) implies $$\alpha_{n}(d\mu) = - \varphi_{n}^{\dag}(z;d\mu)\kappa_{n}^{L/2}= - \varphi_{n}^{\dag}(z;d\nu)\kappa_{n}^{L/2} = \alpha_{n}(d\nu),$$
    and thus by the inverse recursions (\ref{monic_recursion_inv_1}) and (\ref{monic_recursion_inv_2}), we have by iteration that (i) and (ii) hold.\\
    (i) implies (iii) because $\varphi_{j}(z;d\mu)$ and $\mu_{1}, \cdots, \mu_{j-1}$ determine $\mu_{j}(d\mu)$ via $$\int \Phi^{R}_{j}(z)d\mu=0.$$
\end{proof}
\noindent Turning to the weak convergence of $d\mu^{(n)}$:
\begin{theorem}
$d\mu^{(n)}$ is a probability measure on $\mathbb{T}$ for which (\ref{eq:BD_approx_bis}) holds. As $n \rightarrow \infty$, $d\mu^{(n)}\rightarrow d\mu$ weakly.
\end{theorem}
\begin{proof}
By using (iii) of Theorem (\ref{geronimus_thm}), for $j=0,1,\cdots,N$, we have $$\mu_{j}(d\mu^{(n)})=\mu_{j}(d\mu).$$
This equation and its conjugate imply that for any Laurent polynomial $f$ (i.e., a polynomial in $z$ and $z^{-1}$), we have
\begin{equation} \label{eq:coef_equality}
  \lim_{N \rightarrow \infty} \int f(e^{i\theta})d\mu^{(n)} = \int f(e^{i\theta})d\mu,
\end{equation}
because the left-hand side is equal to the right-hand side for large enough $N$. Since Laurent polynomials are dense in $C(\mathbb{T})$, equation (\ref{eq:coef_equality}) holds for all $f$. In other words, we have weak convergence.
\end{proof}

\section{The Szeg\H{o} Limit Theorem}

In this section, the main goal is to prove the Szeg\H{o} limit theorem for operator orthogonal polynomials. We emphasize that while operator-valued Szegő limit theorems have been obtained previously via Toeplitz operator factorization, the proof presented here is entirely based on operator-valued orthogonal polynomials and Schur complements.\\

Derevyagin, Holtz, Khrushchev, and Tyaglov (\cite{DerHK} Thm.28, \cite{Sim2} Thm 2.13.5) extended Szeg\H{o}'s theorem to the matrix version. Using their notation, det and tr representing the determinant and trace, respectively, the following limit holds:
\begin{equation} \label{KSz}
\lim_{n \rightarrow{ \infty}}\frac{\mathrm{det}(T_{n})}{\mathrm{det}(T_{n-1})}=\det \prod_0^{\infty} (I - {\alpha}_k {\alpha}_k^{\dag})
= \exp \int \hbox{tr} \ \log \mu'(\theta) d \theta /2 \pi .
\end{equation}
This is known as the Kolmogorov-Szeg\H{o} formula (see e.g. \cite{Bin1}, \S 4). Szeg\H{o}'s condition, which requires that the right-hand side of the equation be positive, holds if and only if
$$
\sum_0^{\infty} || {\alpha}_k^{\dag} {\alpha}_k || < \infty, 
$$
(extending early work of Delsarte, Genin and Kamp \cite{DelGK}).  \\
\noindent The {\it product theorem for determinants} in (\ref{KSz}) above is simple linear algebra in finitely many dimensions, and holds quite generally.\\

Before proving our results in the operator case, let us review some definitions related to infinite determinants \cite{Sim1}. For an operator $A$ that is trace class, the determinant $\mathrm{det}(I-A)$ is well-defined and can be written as
$$
\det(I-A)=\prod_{j=1}^{\infty}(1-\lambda_{j}),
$$ 
where $(\lambda_{j})_{j=1}^{\infty}$ are the eigenvalues of $A$.
On the other hand, if $A$ is a Hilbert-Schmidt operator, we define the second regularized determinant as
\[
\mathrm{det}_{2}(I-A)=\mathrm{det}[(I-A)e^{A}].
\] 
This is based on the observation that $(I-A)e^{A}-I$ is trace class. It is worth noting that when $A$ is trace class, we have
\[
\mathrm{det}_{2}(I-A)=\det(I-A)e^{tr(A)}.\]
Moreover, for two Hilbert-Schmidt operators $A$ and $B$, we have
\begin{equation} \label{eq:reg_det}
     \mathrm{det}_{2}(I-A)\mathrm{det}_{2}(I-B)=\mathrm{det}_{2}[(I-A)(I-B)]e^{tr(AB)}.
\end{equation} 
Let $\mathcal{W}_{1}(\mathcal{H})$, the Wiener algebra over the trace class operators on $\mathcal{H}$, be the set of all operator-valued functions $G$ on $\mathbb{T}$ of the form 
\begin{equation}\label{eq:W1_series_expansion}
G(z)=\sum_{n=-\infty}^{\infty}z^{n}G_{n}, ~~~ z \in \mathbb{T},
\end{equation} 
where $G_{n}$ is a trace class operator on $\mathcal{H}$ for each $n$ and 
\begin{equation}\label{eq:W1_space}
    \sum_{n=-\infty}^{\infty}||G_{n}||_{1} < \infty. 
\end{equation}

In the following theorem, we will prove the Szeg\H{o} Limit Theorem of the Szeg\H{o} theory of orthogonal polynomials and we assume $\mu_{0} = I$.

\begin{theorem}
Let $T^{R}_{n}$ be the Toeplitz operator matrix and $(\alpha_n)_{n}$ be the Verblunsky coefficients of $\mu$.  Then 
$$
\lim_{n \rightarrow{ \infty}}\frac{\mathrm{det}_{2}(T^{R}_{n})}{\mathrm{det}_{2}(T^{R}_{n-1})}=\det \prod_{k=0}^{\infty} (I - {\alpha}_k {\alpha}_k^{\dag}).
$$
\end{theorem}
\begin{proof}
    Using (\ref{eq:reg_det}), it follows that 
    \begin{eqnarray*}
        \mathrm{det}_{2}(T^{R}_{n}) &=& \mathrm{det}_{2}(T^{R}_{n-1})\mathrm{det}_{2}(I - \nu^{*}_{n} T^{-R}_{n}\nu_{n}) \exp({\mathrm{tr}(\nu^{*}_{n} T^{-R}_{n}\nu_{n})})\\
        &=&\mathrm{det}_{2}(T^{R}_{n-1})\mathrm{det}(I - \nu^{*}_{n} T^{-R}_{n}\nu_{n})\\
        &=& \mathrm{det}_{2}(T^{R}_{n-1}) \mathrm{det}(\kappa^{-R}_{n})=\mathrm{det}_{2}(T^{R}_{n-1}) \mathrm{det}(\rho_{0}^{R} \dots \rho_{n-1}^{R})^{2},     
    \end{eqnarray*}
    which leads to $$\mathrm{det}_{2}(T^{R}_{n})/\mathrm{det}_{2}(T^{R}_{n-1}) = \det \prod_{k=0}^{n-1} (I - {\alpha}_k {\alpha}_k^{\dag}). $$
    Since $T^{R}_{n}$ is definite positive, the coefficients $\alpha_{k}$ are Hilbert-Schmidt and strict contractions. This implies that ${\alpha}_k {\alpha}_k^{\dag}$ are trace class operators, and thus $\det (I - {\alpha}_k {\alpha}_k^{\dag})$ are well defined.
    Letting $n \rightarrow \infty$, the statement follows.
\end{proof}

As in Gohberg and Kaashoek \cite{GohK}, we assume in the theorem below that the measure $\mu'$ is in the Wiener algebra $\mathcal{W}_{1}(\mathcal{H})$.

\begin{theorem}
Assume $\mu'(z)=\sum_{j=-\infty}^{+\infty}\mu_{j}z^{j}\in I + \mathcal{W}_{1}(\mathcal{H})$ for $z \in \mathbb{T}$. Then 
 \begin{equation} \label{szego_thm}
      \lim_{n \rightarrow{ \infty}}\frac{\mathrm{det_{2}}(T^{R,L}_{n})}{\mathrm{det_{2}}(T^{R,L}_{n-1})}
 = \exp \frac{1}{2\pi}\int_{-\pi}^{\pi} \log(\det(\mu'(\theta))d\theta. 
 \end{equation}
\end{theorem}

\begin{proof}
 Write
    \[ 
    L =
    \begin{pmatrix}
    0 & \cdots & 0 & I\\
    0 &  \cdots & I & 0\\
    \vdots & I & \vdots & \vdots\\
    I & 0 & \cdots & 0
    \end{pmatrix},
    \qquad
    \nu_{n}=
    \begin{pmatrix}
    \mu_{-1}\\
    \mu_{-2}\\
    \vdots \\
    \mu_{-n}
    \end{pmatrix}.
    \]
    Then as $T^{R}_{n} = L T^{L}_{n} L$,
    \begin{eqnarray*}
    \Phi_{n}^{L,*} (z)
    &=& z^{n}(z^{-n}I-[I \ z^{-1}I \ \cdots \ z^{-n+1}I] \ T^{-L}_{n}{\nu}_{n}\\
    &=& I - [I~~zI~~ \cdots ~~ z^{n}I]L T^{-L}_{n} L \phi\\
    &=& I - [I~~zI~~ \cdots ~~ z^{n}I]T^{-R}_{n} \phi,\\
    &=& I + F_n(z),
    \end{eqnarray*}
    where $F_n(z) =-[I~~zI~~ \cdots ~~ z^{n}I]T^{-R}_{n} \phi$ and let $F(z) = \sum_{n=1}^{\infty}z^{n}F_{n}$.  \\
    
    Here, we can easily prove that the operator polynomial $F$ is trace class (since its coefficients are the sum of the product of two positive definite Hilbert-Schmidt operators). So $\det \Phi_{n}^{L,*} (z)$ and $\det (\Phi_{n}^{L,*})^{\dag} (z)$ are well-defined.\\
    
    Also, using the inequality (\ref{eq:W1_space}), we can conclude that the series on the right-hand side of (\ref{eq:W1_series_expansion}) converges in the trace class norm. Consequently, if we define $\mu'(\cdot) = I - \tilde{\mu}(\cdot)$ with $\tilde{\mu}$ belonging to the Wiener class $\mathcal{W}_{1}(\mathcal{H})$, then $\tilde{\mu}(z)$ is a trace class operator for all $z \in \mathbb{T}$. This means that $\det (\mu'(z))$ is well-defined for each $z \in \mathbb{T}$, and in particular, the expression $\det (\mu'(z))$ on the right-hand side of (\ref{szego_thm}) is well-defined.\\
    
    From (\ref{eq:BS_measure}) $\mu'^{
    (n)}=[\Phi_{n}^{L,*} \kappa_{n}^{-R}(\Phi_{n}^{L,*})^{\dag}]^{-1}$, and passing to the limit, we have $$\kappa_{\infty}^{R} = (I + F(z))^{\dag} \mu'(z) (I + F(z)).$$    
   It follows that  $$\log \det \kappa_{\infty}^{R}=
    M + 
    \frac{1}{2\pi}\int_{-\pi}^{\pi}\log \det (\mu'(\theta)) d\theta + \Bar{M},
    $$
    where 
    $$
    M = \frac{1}{2\pi}\int_{-\pi}^{\pi}\log \det (I + F(e^{it}))dt, \qquad
    \Bar{M} = \frac{1}{2\pi}\int_{-\pi}^{\pi}\log \det (I + F(e^{it}))^{\dag}dt.
    $$
    So $\det (I + F(\cdot)) $ is analytic on $|z|<1$, and $\det \Phi_{n}^{L,*}(\theta) $ is continuous and non-zero on $|z|\leq 1$. Combining, $\log \det (I + F(\cdot)) $ is analytic on $|z|<1$ and continuous on $|z|\leq 1$. Now Cauchy's theorem gives 
    $$
    M = \frac{1}{2\pi}\int_{-\pi}^{\pi}\log \det (I + F(e^{it}))dt = 0.
    $$
    Similarly for $\bar{M}=0$. Thus, we have the result.

\end{proof}

\section*{Declarations}
\ni \textbf{Funding} The authors received no specific funding for this research.\\

\ni \textbf{Conflict of interest} All authors declare that they have no conflict of interest

\end{document}